\documentclass{amsart}

\usepackage{comment}
\usepackage{pinlabel}  
\usepackage{multicol} 
\usepackage{tikz}
\usepackage{color}
\usepackage{color,soul}
\usepackage[all]{xy}
\usepackage{hyperref}       

\usepackage{float}

\usepackage{booktabs} 

\newcommand{\ulg}{u.l.g.}
\newcommand{\br}{\mathop{br}}

\theoremstyle{plain}
\newtheorem{thm}{Theorem}[section]
\newtheorem{lem}[thm]{Lemma}
\newtheorem{cor}[thm]{Corollary}
\newtheorem{prop}[thm]{Proposition}

\theoremstyle{definition}

\newtheorem{ex}[thm]{Example}

\newtheorem{dfn}[thm]{Definition}

\author[E.~Fink]{Elisabeth Fink}
\address[Elisabeth Fink]{Department of Computer Science, Middlesex University London, Hendon Town Hall, The Burroughs, London, NW4 4BG, United Kingdom}
\email{e.fink@mdx.ac.uk}
\urladdr{https://cs.mdx.ac.uk/dr-elisabeth-fink/}

\author[K.~Zaynullin]{Kirill Zaynullin}
\address[Kirill Zaynullin]{Department of Mathematics and Statistics, University of Ottawa, 150~Louis-Pasteur, Ottawa, ON, K1N 6N5, Canada}
\email{kirill@uottawa.ca}
\urladdr{https://kirillmath.ca}


\begin{document}

\title[Uniquely labeled geodesics of Coxeter groups]{Uniquely labeled geodesics of Coxeter groups}

\begin{abstract}
Studying geodesics in Cayley graphs of groups has been an active area of research for several decades. In this paper we introduce the notion of a \emph{uniquely labeled geodesic} (\ulg). We begin by investigating these objects in finite Coxeter groups of type \(A_n\). In this setting, we define an associated generating function, which allows us to determine precisely the number of \ulg's of a given length together with their corresponding label combinations. These results extend and unify several earlier results concerning unique geodesics in Coxeter groups.

In the second part of the paper, we turn to infinite Coxeter groups whose Coxeter--Dynkin diagrams are simply laced trees. We prove that every \ulg\ of finite branching index has finite length. We then use the affine Coxeter group \(\widetilde{D}_6\) to demonstrate the existence of infinite \ulg's in groups which nevertheless admit no infinite unique geodesics. 

Finally, we provide a detailed description of the geometry of these infinite \ulg's and analyze the way they interact and fellow travel within the Cayley graph of \(\widetilde{D}_6\).
\end{abstract}

\maketitle

\section{Introduction}

Coxeter groups appear naturally as groups generated by reflections (see, for example,~\cite{Hu90}). Over the past several decades, they have attracted substantial interest from many areas of mathematics. Their rich combinatorial and geometric structure has led to deep connections with Lie theory, algebraic geometry, representation theory, and topology. 

 In the present paper we focus on the study of 
geodesics in the Cayley graphs of Coxeter groups --- that is, shortest paths between vertices corresponding to reduced words. This topic has been the subject of extensive investigations since 80's; see, for example, \cite{St84,BH93,He94,Ed95,St97,El97,EE10,LP13,HNW}.

Regarding the counting and classifying of geodesics, in \cite{AC13} and \cite{CK16}, the authors introduce generating series whose coefficients count geodesics in right-angled and even Coxeter groups associated with trees. The paper \cite{MT13} investigates the relationship between geodesics and quasi-geodesics. The problem of uniqueness of geodesics was studied in \cite{Ha17}. The recent paper \cite{HR25} studies prefixes of geodesics and proves that Coxeter elements have a canonical involution decomposition. In \cite{Mi25} the author provides explicit palindromic geodesics for all reflections in finite Weyl groups. The paper \cite{Ha25} studies certain canonical forms of geodesics for affine Weyl groups.

As for the problem of geodesic growth, growth series (in the sense of \cite{P90}) for various classes of Coxeter groups have been studied in  \cite{CD91}, \cite{M03}, and \cite{A04}, and further generalized in \cite{GN97}. These results establish deep connections between the combinatorial structure of Coxeter groups, the geometry of their Cayley graphs, and the asymptotic properties of group growth.

In this paper we introduce the notion of a \emph{uniquely labeled geodesic}.  Rather than considering only the existence and uniqueness of geodesics, we investigate geodesics that are uniquely determined by their total label combination along a path leading to a fixed element of the Cayley graph. In this sense, our approach refines and extends the notion of a unique geodesic (see e.g. \cite{Ha17}).
For example, in an abelian group the only uniquely labeled geodesics are powers of individual generators, since any reduced word involving more than one generator can be reordered without changing either the represented element or the multiset of labels. More generally, every unique geodesic is automatically uniquely labeled, since there is only one reduced expression reaching the corresponding group element. However, the converse need not hold, making uniquely labeled geodesics a broader and more flexible class of objects.

Geometrically, one can show that a word in the generators of a Coxeter group can be a \ulg\ only if it corresponds to a connected path in the associated Coxeter--Dynkin diagram, viewed as a graph encoding the relations of the group. Consequently, the problem of constructing a \ulg\ may be interpreted as a graph-theoretic problem: namely, finding a connected path in a graph that visits each vertex a prescribed number of times.
Our main goal is
to construct infinite \ulg's --- and hence infinite reduced words --- for certain Coxeter groups whose Coxeter--Dynkin diagrams are simply laced trees. The examples we provide exhibit rich geometric properties and demonstrate that infinite uniquely labeled geodesics can occur even in situations where the underlying graph structure is highly constrained.

The paper is organized as follows.
In the first part of the paper, we investigate finite Coxeter groups of type \(A_n\). We introduce an associated generating function, namely a formal power series in \(n\) variables, in which each monomial encodes a particular label combination of a \ulg\  We then derive an explicit formula for every coefficient of this series in terms of the corresponding monomial (Corollary~\ref{cor_coeff}). Using these formulas, we determine both the exact number of nonzero coefficients (Corollary~\ref{cor_nonzero_coeff}) and the total number of uniquely labeled geodesics in type \(A_n\) (Theorem~\ref{thm:numberA}), expressed explicitly as functions of \(n\).

In the second part of the paper, we extend our investigation to infinite Coxeter groups. We introduce the notion of a \emph{branching index} (Definition~\ref{def:finbr}), which, roughly speaking, measures the turning behavior of a connected path in the Coxeter--Dynkin diagram. Using this concept, we prove that every \ulg\ of finite branching index must have finite length (Theorem~\ref{thm:finitebr}).

We then focus on the affine Coxeter group \(\widetilde{D}_6\). In this setting, we construct an infinite periodic \ulg\ (Theorem~\ref{thm:ulg}). Furthermore, we show that there exist three additional infinite \ulg's of a similar type and investigate the geometric relations between these geodesics inside the Cayley graph (Theorem~\ref{thm:fourgeodesics}). In particular, these geodesics exhibit a remarkable fellow-travelling behavior and together form a highly structured family of infinite reduced words.

\medskip

\paragraph{\it Acknowledgements} K.Z. was partially supported by the NSERC Discovery Grant RGPIN-2022-03060.

\section{Uniquely labeled geodesics}

Let \(W\) be a Coxeter group, given by the presentation
\[
W=\langle s_1,\ldots,s_n \mid s_i^2=1,\; (s_is_j)^{m_{ij}}=1,\; i\neq j \rangle,
\]
where \(m_{ij}\) are the Coxeter exponents. Let \(\Gamma\) denote the Cayley graph of \(W\) with respect to the set of simple generators \(\{s_1,\ldots,s_n\}\). By definition, the vertices of \(\Gamma\) are the elements \(w\in W\), and two vertices \(w\) and \(w'\) are connected by an edge labeled by \(i\) whenever
\[
w'=ws_i \quad \text{and} \quad \ell(w')=\ell(w)+1,
\]
where \(\ell\colon W\to \mathbb{Z}_{\ge 0}\) is the length function.

For any \(w\in W\), a shortest path \(\rho(w)\) from the identity element \(1\) to \(w\) corresponds to a reduced expression
\[
w=s_{i_1}\cdots s_{i_l},
\]
and its length \(l\) is equal to \(\ell(w)\). Such a path will be called a \emph{geodesic}.

Rather than labeling an edge of the Cayley graph by the index \(i\) of the corresponding generator, we label it by the \(i\)th standard basis vector. The total \emph{label} of a geodesic \(\rho(w)\), denoted \(\iota(w)\), is defined to be the sum of the labels of its edges. Equivalently,
\[
\iota(w)=(q_1,\ldots,q_n),
\]
where \(q_k\) is the number of occurrences of the generator \(s_k\) in the reduced expression \(\rho(w)\).

\begin{ex}
Suppose that, in the reduced expression
\[
w=s_{i_1}\cdots s_{i_l},
\]
all generators commute pairwise. Then all geodesics \(\rho(w)\) have the same total label.
\end{ex}

In these notes, we are interested in the opposite extreme.

\begin{dfn}
A geodesic \(\rho(w)\) is called a \emph{uniquely labeled geodesic} (\emph{\ulg}) \ if there exists exactly one geodesic from \(1\) to \(w\) with label \(\iota(w)\). Equivalently, these are precisely the elements \(w\in W\) for which the label \(\iota(w)\) determines a unique reduced expression.
\end{dfn}

We define the \emph{generating function}
\[
g_\Gamma=\sum_{\iota} a_{\iota}\, t^\iota,
\]
where the sum ranges over all labels \(\iota=(q_1,\ldots,q_n)\), and \(t^\iota=t_1^{q_1}\cdots t_n^{q_n}\). The coefficient \(a_{\iota}\) denotes the number of uniquely labeled geodesics with label \(\iota\), that is,
\[
a_{\iota}
=
\bigl|\{\, w\in W \mid \rho(w)\text{ is a \ulg\ and }\iota(w)=\iota\,\}\bigr|.
\]
\begin{ex}
For the symmetric group
\[
S_3=\langle s_1, s_2 \mid s_1^2=s_2^2=1,\; (s_1s_2)^3=1\rangle,
\]
the generating function of uniquely labeled geodesics is
\[
g_\Gamma
=
1+t_1+t_2+2t_1t_2+t_1t_2^2+t_1^2t_2.
\]
\end{ex}

\begin{ex}
The Cayley graph $\Gamma$ of the symmetric group 
\[S_4=\langle s_1, s_2, s_3 \mid s_i^2=1,\; (s_1s_2)^3=1,\; (s_2s_3)^3=1,\; (s_1s_3)^2=1\rangle\] has the form 
\[\tiny
\xymatrix@R=3.5em @C=1.2em{
 & & & 1 \ar[d]|-{2}\ar[dl]|-{1}\ar[dr]|-{3}  & & & \\
 &&s_1 \ar[dl]|-{2} \ar[dr]|-(.75){3} & s_2\ar[dl]|-(.75){1}\ar[dr]|-(.75){3}  & s_3\ar[dl]|-(.75){1}\ar[dr]|-{2} && \\
 &s_1s_2\ar[dl]|-{3} \ar[d]|-{1}& s_2s_1\ar[dl]|-{2} \ar[ddr]|-{3}&  s_1s_3\ar[d]|-{2}  & s_2s_3\ar[ddl]|-{1} \ar[dr]|-{2} & s_3s_2 \ar[d]|-{3}\ar[dr]|-{1} & \\ 
 s_1s_2s_3\ar[ddr]|-{1}\ar[ddrr]|-(.75){2} & s_1s_2s_1\ar[dd]|-(.75){3} &  &  s_1s_3s_2\ar[ddl]|-{3}\ar[ddr]|-{1}  & & s_2s_3s_2 \ar[dd]|-(.75){1}& s_3s_2s_1\ar[ddl]|-{3}\ar[ddll]|-(.75){2} \\ 
 & & &  s_2s_1s_3\ar[d]|-{2}  & & & \\ 
 & s_1s_2s_3s_1\ar[rd]|-{2} & s_1s_2s_3s_2\ar[rd]|-(0.75){1} &  s_2s_1s_3s_2\ar[dr]|-(0.75){1}\ar[dl]|-(0.75){3}  & s_3s_2s_1s_2\ar[dl]|-(0.75){3} & s_3s_2s_1s_3\ar[dl]|-{2} & \\
 && s_1s_2s_3s_1s_2\ar[rd]|-{1} & s_1s_3s_2s_3s_1\ar[d]|-{2} & s_3s_2s_1s_3s_2 \ar[dl]|-{3} && \\
 &&& s_1s_2s_3s_1s_2s_1 &&&
 }
 \]
 which gives 
 \begin{align*}
 g_\Gamma &= 1+(t_1+t_2+t_3)+ (2t_1t_2+2t_2t_3)\\
 &+(t_1^2t_2+t_1t_2^2+2t_1t_2t_3+t_2^2t_3+t_2t_3^2)\\
  &+4t_1t_2^2t_3+(t_1^2t_2^2t_3+t_1t_2^2t_3^2+2t_1t_2^3t_3).
 \end{align*}
 Observe that 
 \[
 w=s_1s_2s_3s_1=s_1s_2s_1s_3=s_2s_1s_2s_3
 \]
 has a unique reduced expression with label $(1,2,1)$ so it is counted as a monomial $t_1t_2^2t_3$.

 \end{ex}

\begin{ex}
Let \(\mathcal{G}\) be a finite graph in which any two vertices are connected by at most one edge, and let \(W\) denote the right-angled Coxeter group associated to the complement graph $\mathcal{G}^c$. 
By definition, \(W\) is generated by elements \(s_\alpha\), one for each vertex \(\alpha\) of \(\mathcal{G}\), 
subject to the following relations:
\begin{itemize}
\item \(s_\alpha^2 = 1\) for all \(\alpha\),
\item \(s_\alpha s_\beta = s_\beta s_\alpha\) whenever \(\alpha\) and \(\beta\) are not adjacent, and
\item no relation between \(s_\alpha\) and \(s_\beta\) when \(\alpha\) and \(\beta\) are adjacent.
\end{itemize}

One of the simplest examples is  
the affine Weyl group $W$ (infinite dihedral group) with the affine Dynkin diagram $\mathcal{G}$ of type \(A_1^{(1)}\) or \(A_1^{(2)}\) of \cite[pp.54-55]{Ka90}).

Since no relation is imposed on adjacent generators, the uniquely labeled geodesics in the Cayley graph \(\Gamma\) of $W$ are naturally in bijection with connected paths in \(\mathcal{G}\). Hence, after labeling the vertices of \(\mathcal{G}\) (equivalently, the generators of \(W\)) by \(\alpha_1,\dots,\alpha_n\), the coefficient \(a_{(q_1,\ldots,q_n)}\) of the generating function \(g_\Gamma\) records the number of connected paths in \(\mathcal{G}\) that pass through \(\alpha_j\) exactly \(q_j\) times for each \(j\).
\end{ex}

\begin{ex}\label{ex:aff} 
Let \(W\) be the Weyl group of affine Dynkin type \(A_2^{(1)}\), given by
\[
W=\langle s_1,s_2,s_3 \mid s_i^2=1,\ (s_is_j)^3=1 \text{ for } i\ne j\rangle.
\]
The corresponding Coxeter--Dynkin diagram \(\mathcal{G}\) is a triangle,
\[\tiny
\begin{tikzpicture}[scale=1, baseline=(current bounding box.center)]
\node (a3) at (0,1.2) {\(\alpha_3\)};
\node (a1) at (-1.2,0) {\(\alpha_1\)};
\node (a2) at (1.2,0) {\(\alpha_2\)};
\draw (a1) -- (a2);
\draw (a1) -- (a3);
\draw (a2) -- (a3);
\end{tikzpicture}
\]
whose edges represent the braid relations \((s_is_j)^3=1\). In particular, no pair of simple generators commute. The infinite word
\[
(s_1s_2s_3)^\infty
\]
has the property that each of its finite connected subwords is a unique geodesic in the Cayley graph $\Gamma$, and therefore defines a \ulg\   As a consequence, the generating function \(g_\Gamma\) is a formal power series, rather than a polynomial, in \(\mathbb{Z}[[t_1,t_2,t_3]]\). This example was investigated extensively in \cite{LP13} and is closely related to the classical problem of constructing infinite reduced words.
\end{ex}

Let \(\Theta \subset \{s_1,\ldots,s_n\}\) be a subset of simple generators, and let \(W_\Theta\) denote the subgroup of \(W\) generated by \(\Theta\), subject to the induced relations. Such a subgroup is called a \emph{parabolic subgroup} of \(W\). By definition it is also a Coxeter group.

Consider the set of right cosets \(W_\Theta \backslash W\). Each coset contains a unique representative of minimal length. We denote the set of all such minimal length representatives by \(W^\Theta\).
Let \(\Gamma_\Theta\) be the Hasse diagram whose vertices are the elements of \(W^\Theta\), with edges labeled by \(i\) whenever
\[
W_\Theta w' = W_\Theta w s_i
\qquad \text{and} \qquad
l(w') = l(w) + 1.
\]
Equivalently, \(\Gamma_\Theta\) is the induced subgraph of \(\Gamma\) on the vertex set \(W^\Theta\), endowed with the partial order coming from the right weak Bruhat order.

If \(w \in W^\Theta\), then any geodesic \(\rho(w)\) in \(\Gamma\) from \(1\) to \(w\) lies entirely in \(W^\Theta\), and hence is also a geodesic in \(\Gamma_\Theta\). Consequently, for \(w \in W^\Theta\), the path \(\rho(w)\) is a u.l.g\! in \(\Gamma_\Theta\) if and only if it is a \ulg\! in \(\Gamma\).

Let \(g_{\Gamma_\Theta}\) denote the generating function counting uniquely labeled geodesics in \(\Gamma_\Theta\), and let \(a^{\Theta}_{\iota}\) denote its coefficients. Then for all labels \(\iota\) we have
\[
a^{\Theta}_{\iota} \le a_{\iota}.
\]

\begin{ex}
Let \(W=S_4\), and let \(\Theta=\{s_1,s_3\}\). Then
\[
W_\Theta=\langle s_1,s_3\rangle \simeq S_2\times S_2.
\]
The set of minimal length representatives of the right cosets \(W_\Theta\backslash W\) is
\[
W^\Theta=\{1,\ s_2,\ s_2s_1,\ s_2s_3,\ s_2s_1s_3,\ s_2s_3s_1\}.
\]

The Hasse diagram \(\Gamma_\Theta\) is given by
\[
\tiny
\xymatrix@=2.2em{
 & & 1 \ar[d]|-{2} & & \\
 & & s_2 \ar[dl]|-{1} \ar[dr]|-{3} & & \\
 & s_2s_1 \ar[d]|-{3} & & s_2s_3 \ar[d]|-{1} & \\
 & s_2s_1s_3 & & s_2s_3s_1 &
}
\]

Thus 
\begin{align*}
g_{\Gamma_\Theta}
&=1+t_2+(t_1t_2+t_2t_3)+2t_1t_2t_3.
\end{align*}
\end{ex}

\section{Coxeter groups of type $A$}

We now turn to the study of uniquely labeled geodesics in a finite Coxeter group \(W\) of Dynkin type \(A\), with simple generators \(\{s_1,\ldots,s_n\}\), $n\ge 2$.
In this case, the Coxeter exponents are given by
\[
m_{ij}=2 \quad \text{if } |i-j|>1,
\qquad
m_{ij}=3 \quad \text{if } |i-j|=1,
\]
or, equivalently, the Dynkin diagram $\mathcal{G}$ is the chain
{\tiny\[
\xymatrix{
\alpha_1 \ar@{-}[r] & \alpha_2 \ar@{-}[r] & \cdots \ar@{-}[r] & \alpha_n
}
\]}

Consider the Cayley graph $\Gamma$ of $W$ and the generating function $g_\Gamma$. We start with the following immediate observation.

\begin{lem}\label{lem:shape}
Every nonzero monomial of \(g_\Gamma\) is of the form
\[
t_l^{\,i_l} t_{l+1}^{\,i_{l+1}} \cdots t_m^{\,i_m},
\qquad 1\le l\le m\le n,
\]
where each exponent \(i_l,i_{l+1},\dots,i_m\) is nonzero.
\end{lem}

\begin{proof} We may assume $n\ge 3$.
Suppose that a geodesic (that is, a reduced word) \(\rho(w)\) contains no occurrence of the generator \(s_k\), where \(2\le k\le n-1\), but does contain generators from both subsets
\[
S=\{s_1,\dots,s_{k-1}\}
\qquad \text{and} \qquad
T=\{s_{k+1},\dots,s_n\}.
\]
Then, without loss of generality, \(\rho(w)\) contains a subword \(xy\) with \(x\in S\) and \(y\in T\), so that
\[
w=u\cdot xy\cdot v
\]
for some words \(u\) and \(v\). Since \(x\) and \(y\) commute, we also have
\[
w=u\cdot yx\cdot v.
\]
These two expressions yield two distinct geodesics with the same label. Therefore \(\rho(w)\) cannot be a \ulg
\end{proof}

\begin{cor}\label{cor:conn}
Let \(\rho(w)\) be a \ulg\! with \(l(w)\ge 2\). Then any two adjacent generators appearing in \(\rho(w)\), say
\(
\quad\cdots s_i s_j \cdots\quad
\)
must satisfy \(|i-j|=1\). 

In other words a \ulg\! corresponds to a connected path in the diagram $\mathcal{G}$.
\end{cor}

Observe that the Coxeter group \(W\) admits two natural involutions. 
The first is the Dynkin involution \(\tau\), induced by reversing the indices,
\[
\tau\colon s_k \mapsto s_{n+1-k},
\]
and the second is the inversion map, obtained by reversing words:
\[
w = s_{i_1}s_{i_2}\cdots s_{i_l} \mapsto w^{-1} = s_{i_l}\cdots s_{i_2}s_{i_1}.
\]
Both involutions preserve the property of being a \ulg; namely,
\[
\rho(w)\text{ is a \ulg} \iff \rho(\tau(w))\text{ is a \ulg} \iff \rho(w^{-1})\text{ is a \ulg}
\]
Moreover, the label is preserved by taking the inverse, i.e.,  $\iota(w)=\iota(w^{-1})$.
We say that a property of the word $\rho(w)$ holds up to inversion (resp.~index swap) if it also holds for the word $\rho(w^{-1})$ (resp.~for $\rho(\tau(w))$).

As the next lemma shows, the (long) braid relations impose strong restrictions on being a \ulg
\begin{lem}\label{lem:sublist}
Suppose that $n\ge 3$ and \(\rho(w)\) is a \ulg \ 
Then, up to inversion and index swap, it cannot contain a subword of the form
\[
s_m s_{m-1}\cdots s_{l+1}s_ls_{l+1}\cdots s_{m-1}s_ms_{m-1},\quad\text{ where }\;l+2 \le m.
\]
\end{lem}

\begin{proof} It is enough to show that $\rho(w)$ does not contain such a subword.
Suppose on contrary that \(\rho(w)\) contains such a subword.
Applying the braid relation to the final three terms, we obtain
\[
s_m s_{m-1}\cdots s_{l+1}s_ls_{l+1}\cdots  \underline{s_m} s_{m-1}s_m.
\]
Since the element \(s_m\) commutes with all generators \(s_j\) for \(l \le j \le m-2\), we may move the indicated \(s_m\) to the left, yielding
\[
s_m s_{m-1}\underline{s_m} \cdots s_{l+1}s_ls_{l+1}\cdots s_{m-1}s_m.
\]
Applying the braid relation once more to the first three terms, we obtain
\[
s_{m-1}s_ms_{m-1}\cdots s_{l+1}s_ls_{l+1}\cdots s_{m-1}s_m.
\]
Thus, we have constructed two distinct words that represent the same element and contain the same number of occurrences of each generator. Hence, they have the same label, a contradiction.
\end{proof}

We then obtain the following.

\begin{thm}
Let $\rho(w)$ be a uniquely labeled geodesic with label $\iota=\iota(w)$, and let $t^{\iota}$ denote the respective monomial in $g_\Gamma$.
Then it has to be one of the following types:

\begin{itemize}
\item[I(a)] We have $\rho(w)=s_m$ and $t^\iota=t_m$. 

\item[I(b)] Up to inversion for \(l < m\), the word \(\rho(w)\) is strictly decreasing from \(m\) to~\(l\), namely,
\(
\rho(w) = s_m s_{m-1}\cdots s_{l+1}s_l\)  and \( t^\iota=t_l\cdots t_m
\).

\smallskip

\item[II(a)] Up to index swap for $l<m$,  the word \(\rho(w)\) first decreases from \(m\) to \(l\), and then increases from \(l\) to \(m\), namely, 

\noindent
$\rho(w)=s_m s_{m-1}\cdots s_{l+1}s_l s_{l+1}\cdots s_{m-1}s_m$ and $ t^\iota=t_l(t_{l+1}\cdots t_m)^2$.

\smallskip
This may be illustrated as
\(\tiny
\begin{tikzpicture}[scale=0.6, baseline=(current bounding box.center), >=stealth]
\node (si) at (-1.5,1.2) {\(s_m\)};
\node (sl) at (0,0) {\(s_l\)};
\node (sm) at (1.5,1.2) {\(s_m\)};
\draw[->] (si) -- (sl);
\draw[->] (sl) -- (sm);
\end{tikzpicture}
\)
\smallskip

\item[II(b)] Up to inversion and index swap for $l<i<m-1$ ($n\ge 4$), the word \(\rho(w)\) first decreases from \(i\) to \(l\), and then increases from \(l\) to \(m\), namely,

\noindent
$
\rho(w)=s_i s_{i-1}\cdots s_{l+1}s_l s_{l+1}\cdots s_{m-1}s_m$ and $t^\iota=t_l\bigl(t_{l+1}\cdots t_i\bigr)^2 t_{i+1}\cdots t_m$.

\smallskip
This may be illustrated as
\(\tiny
\begin{tikzpicture}[scale=0.6, baseline=(current bounding box.center), >=stealth]
\node (si) at (-1.5,1.2) {\(s_i\)};
\node (sl) at (0,0) {\(s_l\)};
\node (sm) at (1.4,2.0) {\(s_m\)};
\draw[->] (si) -- (sl);
\draw[->] (sl) -- (sm);
\end{tikzpicture}
\)
\smallskip

\item[II(c)] Up to inversion and index swap for $l<m-1$ ($n\ge 3$), the word \(\rho(w)\) first decreases from \(m-1\) to \(l\), and then increases from \(l\) to \(m\), namely, 

\noindent
$\rho(w)=s_{m-1}\cdots s_{l+1}s_l s_{l+1}\cdots s_{m-1}s_m$ and  $t^\iota=t_l\bigl(t_{l+1}\cdots t_{m-1}\bigr)^2 t_{m}$.

\smallskip
This may be illustrated as
\(\tiny
\begin{tikzpicture}[scale=0.6, baseline=(current bounding box.center), >=stealth]
\node (si) at (-1.5,1.2) {\(s_{m-1}\)};
\node (sl) at (0,0) {\(s_l\)};
\node (sm) at (1.4,1.4) {\(s_m\)};
\draw[->] (si) -- (sl);
\draw[->] (sl) -- (sm);
\end{tikzpicture}
\)
\smallskip

\item[III(a)] Up to inversion for \(l < i<j-1 < m-1\) ($n\ge 5$), the word \(\rho(w)\) first decreases from \(i\) to the minimal index \(l\), then increases from \(l\) to the maximal index \(m\), and finally decreases from \(m\) to \(j\), namely, 

\noindent
$\rho(w)=
s_i s_{i-1}\cdots s_{l+1}s_l s_{l+1}\cdots s_{m-1}s_m s_{m-1}\cdots s_{j+1}s_j$ and 

\noindent
$t^\iota=t_l(t_{l+1}\cdots t_i)^2t_{i+1}\cdots t_{j-1}(t_j\cdots t_{m-1})^2t_m$.

\smallskip
This may be illustrated as
\(\tiny
\begin{tikzpicture}[scale=0.5, baseline=(current bounding box.center), >=stealth]
\node (si) at (-2,1.2) {\(s_i\)};
\node (sl) at (0,0) {\(s_l\)};
\node (sm) at (2,3) {\(s_m\)};
\node (sj) at (3.8,1.7) {\(s_j\)};
\draw[->] (si) -- (sl);
\draw[->] (sl) -- (sm);
\draw[->] (sm) -- (sj);
\end{tikzpicture}
\)
\smallskip

\item[III(b)] Up to inversion
 for \(l < j\le i < m\) ($n\ge 3$), the word \(\rho(w)\) first decreases from \(i\) to the minimal index \(l\), then increases from \(l\) to the maximal index \(m\), and finally decreases from \(m\) to \(j\), namely, 

\noindent
$\rho(w)=s_i s_{i-1}\cdots s_{l+1}s_l s_{l+1}\cdots s_{m-1}s_m s_{m-1}\cdots s_{j+1}s_j$ and 

\noindent
$t^\iota=t_l(t_{l+1}\cdots t_{j-1})^2(t_{j}\cdots t_{i})^3(t_{i+1}\cdots t_{m-1})^2t_m$.

\smallskip
This may be illustrated as
\(\tiny
\begin{tikzpicture}[scale=0.5, baseline=(current bounding box.center), >=stealth]
\node (si) at (-2,2.5) {\(s_i\)};
\node (sl) at (0,0) {\(s_l\)};
\node (sm) at (2,3) {\(s_m\)};
\node (sj) at (3.8,0.3) {\(s_j\)};
\draw[->] (si) -- (sl);
\draw[->] (sl) -- (sm);
\draw[->] (sm) -- (sj);
\end{tikzpicture}
\)
\smallskip

\item[III(c)] Up to inversion
 for \(l < i < m-1\) ($n\ge 4$), the word \(\rho(w)\) first decreases from \(i\) to the minimal index \(l\), then increases from \(l\) to the maximal index \(m\), and finally decreases from \(m\) to \(i+1\), namely, 

\noindent
\(
s_i s_{i-1}\cdots s_{l+1}s_l s_{l+1}\cdots s_{m-1}s_m s_{m-1}\cdots s_{i+1}\) and  $t^\iota=t_l(t_{l+1}\cdots t_{m-1})^2t_m$.

\smallskip
This may be illustrated schematically as follows:
\[\tiny
\begin{tikzpicture}[scale=0.5, baseline=(current bounding box.center), >=stealth]
\node (si) at (-2,1.5) {\(s_i\)};
\node (sl) at (0,0) {\(s_l\)};
\node (sm) at (2,3) {\(s_m\)};
\node (sj) at (3.8,1.6) {\(s_{i+1}\)};
\draw[->] (si) -- (sl);
\draw[->] (sl) -- (sm);
\draw[->] (sm) -- (sj);
\end{tikzpicture}
\]

\end{itemize}
\end{thm}

\begin{proof}
By Lemma \ref{lem:sublist}, the connected path $\rho(w)$ must avoid the subwords of the form
\[
s_i s_{i-1}\cdots s_{l+1}s_l s_{l+1}\cdots s_{m-1}s_m s_{m-1}\cdots s_{j+1}s_j
\]
where either $i\ge m$ or $j\le l$. This immediately restricts us to the desired word-types. Counting numbers of each generator appearing in the respective word gives the desired expression for the monomial.
\end{proof}

\begin{cor}\label{cor:maxchain}
For \(n\ge 3\), the maximal length of a \ulg\ is \(3n-4\).
\end{cor}

\begin{proof}
The longest word arises in the case {III(b)}, with
\[
i=n-1,\qquad l=1,\qquad m=n,\qquad j=2.
\]
Thus, it is of the form
\[
s_{n-1}s_{n-2}\cdots s_2s_1s_2\cdots s_{n-1}s_ns_{n-1}\cdots s_2.
\]
Its length is \(3n-4\), and it is a \ulg\! with the label
\( (1,3,3,\dots,3,3,1) \).
\end{proof}

\begin{cor}\label{cor_coeff}
The coefficient $a_\iota$ of the generating function $g_\Gamma$
 depends only on the type of the respective word and monomial and is given by
\[
a_{\iota}=
\left\{
\begin{array}{ll}
1, & \text{for type I(a)},\ \text{or type II(a)},\\[1ex]
2, & \text{for type I(b)},\ \text{or type II(b)},\ \text{or type III(a,b)},\\[1ex]
2(m-l), & \text{for types II(c) and III(c)}.
\end{array}
\right.
\]
\end{cor}

\begin{proof}
For types I(a) and II(a) the \ulg\! $\rho(w)$ is uniquely determined by the respective monomial.

For types I(b), II(b), III(a,b) there are two different geodesics $\rho(w)$ and $\rho(w^{-1})$ corresponding to the same monomial. 

For type II(c) there are 4 different geodesics having the same monomial. For type III(c) there are $2(m-l-2)$ different geodesics
(the choices of $i$ up to inversion).
\end{proof}

\begin{cor}\label{cor_nonzero_coeff}
For $n\geq 2$ the generating function \(g_\Gamma\) has exactly
\[
\tfrac{1}{12}n^4 - \tfrac{1}{6}n^3 + \tfrac{17}{12}n^2 - \tfrac{1}{3}n
\]
nonzero coefficients, excluding the constant term.
\end{cor}

\begin{proof}
For $n\ge 5$ we count the number of distinct monomials appearing in \(g_\Gamma\) by considering separately the monomials of different types described in the theorem.

{I.} There are $n$ monomials of type I(a) and $\binom{n}{2}$ monomials of type I(b) that is $\binom{n+1}{2}$ in total.

{II.} There are $2\binom{n}{2}$ monomials of type II(a), 
$2\binom{n-1}{3}$ monomials of type II(b) and $\binom{n-1}{2}$ monomials of type II(c).

{III.} Finally, there are $\binom{n-1}{4}$ monomials of type III(a), and $\binom{n+1}{4}$ monomials of type III(b).

Therefore, for $n\geq 5$ the total number of monomials is
\[
\textstyle
\binom{n+1}{2}+2\binom{n}{2}+2\binom{n-1}{3}+\binom{n-1}{2}+\binom{n-1}{4}+\binom{n+1}{4}. 
\]
The values $n \leq 4$ are verified by direct computation.
\end{proof} \qedhere

\begin{ex}
We list all non-zero labels of uniquely labeled geodesics for $n=4$:

\medskip

\begin{tabular}{|l|l|}
Types & Labels  \\
\hline
I & \tiny{1000  0100  0010  0001 1100  0110  0011  0111 1110 1111}  \\
II(a) & \tiny{1200 0120 0012 1220 0122 1222 0021 0210 2100 0221 2210 2221}   \\ 
II(b); II(c); III(b) & \tiny{1211 1121; 1210 0121 1221; 1331 1231 1321 1310 0131}   
\end{tabular}
\medskip
\end{ex}

\begin{cor}\label{thm:numberA}
The total number of uniquely labeled geodesics for $n\ge 2$ is given by
\[
g_\Gamma(1,\ldots,1)=\tfrac{1}{6}n^2(n^2+5)+1.
\]
\end{cor}

\begin{proof}
We combine Corollaries \ref{cor_coeff} and \ref{cor_nonzero_coeff}.
Suppose $n\ge 5$. Except for types II(c) and III(c) we have the following 

\medskip

\begin{tabular}{lccc @{\hskip 0.5in} lcc} 
Type & coefficient & monomials & & Type & coefficient & monomials \\ 
\cmidrule(r){1-3} \cmidrule(r){5-7}
I(a)  & $1$ & $n$ & &  II(b) & $2$ & $2\binom{n-1}{3}$\\[0.5ex]
I(b) & $2$ & $\binom{n}{2}$ & & III(a) & $2$ & $\binom{n-1}{4}$\\[0.5ex]
II(a) & $1$ & $2\binom{n}{2}$ & & III(b) & $2$ & $\binom{n+1}{4}$
\end{tabular}
  
\medskip

Therefore the total number of uniquely labeled geodesics in those cases is given by
\[\textstyle
n+4\binom{n}{2}+4\binom{n-1}{3}+2\binom{n-1}{4}+2\binom{n+1}{4} 
\]
As there are $n-k$ monomials $t_l(t_{l+1}\cdots t_{m-1})^2t_m$ of types II(c) and III(c) with $k=m-l$ this adds
\[
\sum_{k=2}^{n-1} 2k(n-k)=\tfrac{1}{3}(n^3-7n+6)
\]
geodesics which gives the desired expression (after simplifying it). The cases for $n\le 4$ are verified directly.
\end{proof}

\begin{cor}(cf. Hart~\cite{Ha17}.)
In the symmetric group \(S_{n+1}\), $n\ge 1$ there are \(n^2+1\) elements admitting a unique geodesic, i.e.,\ a unique reduced expression.
\end{cor}

\begin{proof}
A unique reduced expression corresponds to a word of type~I. There are $n$ words of type~I(a) and $2\binom{n}{2}$ words of type~I(b) (up to inversion).
Hence, the total number of unique reduced expressions is
\(
n+2\binom{n}{2}+1=n^2+1
\).
\end{proof}

%
 

\section{Simply laced trees}
We now turn to the case where the Coxeter--Dynkin diagram \(\mathcal{G}\) of the Coxeter group \(W\) is no longer a chain, as in type \(A\), but a finite simply laced graph, that is, a graph in which any two vertices are connected by at most one edge. Equivalently, the Coxeter exponents satisfy \(m_{ij}\in\{2,3\}\).

More precisely, each vertex \(\alpha\) of \(\mathcal{G}\) corresponds to a generator \(s_\alpha\) of \(W\). If two vertices \(\alpha\) and \(\beta\) are adjacent, then the generators satisfy the braid relation
\[
s_\alpha s_\beta s_\alpha = s_\beta s_\alpha s_\beta,
\]
whereas if \(\alpha\) and \(\beta\) are not adjacent, the generators \(s_\alpha\) and \(s_\beta\) commute.

Let \(\Gamma\) be the Cayley graph of \(W\), and let \(w\in W\). As in the type \(A\) case, to each geodesic \(\rho(w)\) we associate a label \(\iota(w)=(q_\alpha)_{\alpha\in\mathcal{G}}\) and the corresponding monomial
\[
t^{\iota(w)}=\prod_{\alpha\in \mathcal{G}} t_\alpha^{q_\alpha},
\]
where \(q_\alpha\) denotes the number of occurrences of the generator \(s_\alpha\) in the reduced expression \(\rho(w)\).
We then define the generating function
\[
g_\Gamma=\sum_{\iota} a_{\iota}\, t^\iota,
\]
which counts the number of uniquely labeled geodesics in \(\Gamma\).

By arguments analogous to those in type \(A\), we obtain the following generalizations of Lemma~\ref{lem:shape}, Corollary~\ref{cor:conn}, and Lemma~\ref{lem:sublist}.

\begin{lem}
All monomials of \(g_\Gamma\) are of the form
\[
\prod_{\alpha\in \mathcal{G}'} t_\alpha^{q_\alpha},
\]
where \(\mathcal{G}'\) is a connected subgraph of \(\mathcal{G}\) and \(q_\alpha>0\) for all \(\alpha\in \mathcal{G}'\).
\end{lem}

\begin{cor}\label{cor:ulgcomm}
Let \(\rho(w)\) be a \ulg\ with \(\ell(w)\ge 2\). Then any two consecutive generators in a reduced expression,
\[
w=\cdots s_\alpha s_\beta\cdots,
\]
correspond to adjacent vertices \(\alpha,\beta\) in \(\mathcal{G}\).

Equivalently, a \ulg\ in \(\Gamma\) determines a connected path (possibly with returns) in the Coxeter--Dynkin diagram \(\mathcal{G}\).
\end{cor}

\begin{lem}\label{lem:uglw}
A \ulg\ cannot contain a subword of either of the forms
\[
s_\alpha s_\beta\, v\, s_\beta s_\alpha s_\beta
\quad\text{or}\quad
s_\beta s_\alpha s_\beta\, v\, s_\beta s_\alpha,
\]
where the word \(v\) contains no generators adjacent to \(s_\alpha\).
\end{lem}

\begin{proof}
This follows from applying the braid relation to the first and last term:
\[s_\alpha s_\beta\, v\, s_\beta s_\alpha s_\beta =  s_\beta s_\alpha s_\beta s_\alpha\, v\, s_\alpha s_\beta s_\alpha = s_\beta s_\alpha s_\beta \, v \, s_\beta s_\alpha,
\] which has the same label as the initial word, hence the labeling is not unique. The second case can be shown using the same argument.
\end{proof}

\begin{dfn}
A connected path in \(\mathcal{G}\) is called \emph{simple}, or a path \emph{with no returns}, if each vertex occurs exactly once along the path.

A vertex \(\nu\) of a path \(\rho(w)\) is called a \emph{turning vertex} if the corresponding generator \(s_\nu\) appears in a subword of the form \(s_\mu s_\nu s_\mu\) in \(\rho(w)\); that is, the path passes from \(\mu\) to \(\nu\) and then immediately returns to \(\mu\).

To each path \(\rho(w)\) we associate the maximal ordered list \((\nu_1,\ldots,\nu_r)\) of turning vertices, taken in the order in which they occur along the path. We also include the endpoints \(\nu_0\) and \(\nu_{r+1}\), corresponding to the initial and terminal vertices of \(\rho(w)\).

Since \(\mathcal{G}\) is a tree, each subpath \([\nu_i,\nu_{i+1}]\) is simple and passes through \(\nu_i\) and \(\nu_{i+1}\) exactly once. This structure may be illustrated schematically as follows:
{\tiny\[
\xymatrix@=1em{
 & & & s_{\nu_2} \ar[rrddd]|{\cdots} & & & & s_{\nu_r} \ar[ddr] & \\
 s_{\nu_0} \ar[rdd] & & & & & & & & \\
 & & & & & & & & s_{\nu_{r+1}} \\
 & s_{\nu_1} \ar[rruuu] & & & & s_{\nu_{r-1}} \ar[rruuu] & & &
}
\]}
\end{dfn}

\begin{lem}
Let \(\rho(w)\) be a \ulg, and let \((\nu_0,\ldots,\nu_{r+1})\) be the associated list of turning vertices, with the initial and terminal vertices included.

Then, for all \(1\le i\le r\), the subpath \([\nu_{i-2},\nu_{i+2}]\) passes through \(\nu_i\) exactly once. (For notational convenience, we set \(\nu_{-1}=\nu_0\) and \(\nu_{r+2}=\nu_{r+1}\).)
\end{lem}

\begin{proof}
Suppose that the subpath \([\nu_{i-2},\nu_{i+2}]\) contains a second occurrence of \(\nu_i\). Then this occurrence must lie either in \([\nu_{i-2},\nu_{i-1})\) or in \((\nu_{i+1},\nu_{i+2}]\). By symmetry, it suffices to consider the latter case.

Assume therefore that \(\nu_i\) appears again in \((\nu_{i+1},\nu_{i+2}]\). Since \(\nu_i\) is a turning vertex, the path \(\rho(w)\) contains a subword
\[
s_{\mu_i}s_{\nu_i}s_{\mu_i}
\]
for some vertex \(\mu_i\) adjacent to \(\nu_i\).

Now consider the segment of the path from this occurrence of \(\nu_i\) to the next occurrence of \(\nu_i\). Since \(\mathcal G\) is a tree, any path that leaves \(\nu_i\) and later returns to \(\nu_i\) without passing through \(\nu_i\) in between must leave and return through the same adjacent vertex \(\mu_i\). Hence \(\rho(w)\) contains a subword of the form
\[
s_{\mu_i}s_{\nu_i}s_{\mu_i}\, v\, s_{\mu_i}s_{\nu_i},
\]
for some word \(v\).

It remains to show that \(v\) contains no generators adjacent to \(s_{\nu_i}\). Suppose otherwise. Then the corresponding subpath would pass through a vertex adjacent to \(\nu_i\) before reaching the second occurrence of \(\nu_i\). Since \(\mathcal G\) is a tree, this would force an additional turn at \(\nu_i\) or create another turning vertex between \(\nu_i\) and \(\nu_{i+2}\), contradicting the fact that \(\nu_{i+1}\) is the unique turning vertex between \(\nu_i\) and \(\nu_{i+2}\).

Thus \(v\) contains no generators adjacent to \(s_{\nu_i}\), and so \(\rho(w)\) contains a forbidden subword of the form excluded by Lemma~\ref{lem:uglw}. 
\end{proof}

\begin{cor}\label{cor:typesulg}
A \ulg\ \(\rho(w)\) must be of one of the following types.

\paragraph{I}
A simple path; that is, each generator in \(\rho(w)\) occurs exactly once.

\paragraph{II}
A path with a single turn; that is, \(r=1\).

\paragraph{III}
A path with \(r\ge 2\) turns such that, for all \(1\le i\le r\), the subpath \([\nu_{i-2},\nu_{i+2}]\) passes through the turning vertex \(\nu_i\) exactly once.
\end{cor}

Observe that when \(\mathcal{G}\) is a chain, types~I–III coincide with the corresponding types in the type~\(A\) case. In that setting, these conditions are also sufficient for a path to be a \ulg\  In general, however, there exist paths of type~III that are not uniquely labeled geodesics.
 
 \begin{ex}
Consider the Weyl group \(D_5\) with Coxeter--Dynkin diagram \(\mathcal{G}\)
{\tiny\[
\xymatrix{
&& \alpha_2\ar@{-}[d] & \\
\alpha_1 \ar@{-}[r]& \alpha_3 \ar@{-}[r]& \alpha_4 \ar@{-}[r]& \alpha_5 
}
\]}

Consider the reduced word
\[
\rho(w)=s_3s_1s_3s_4s_2s_4s_5s_4s_3s_1.
\]
The corresponding path is of type~III with \(r=3\) and turning vertices
\[
(\alpha_3,\alpha_1,\alpha_2,\alpha_5,\alpha_1).
\]
However, this word contains a subword of the form excluded by Lemma~\ref{lem:uglw}, and hence it is not a \ulg
\end{ex}

We now provide the following equivalent criteria for type III condition.

\begin{lem}\label{lem:symmetric-local}
Let \(\rho(w)\) be a \ulg\  with turning vertices
\[
(\nu_0,\nu_1,\ldots,\nu_r,\nu_{r+1}),
\]
where \(\nu_0\) and \(\nu_{r+1}\) are the initial and terminal vertices. For each \(1\le i\le r\), let \(\mu_i\) be the unique vertex adjacent to \(\nu_i\) such that
\[
s_{\mu_i}s_{\nu_i}s_{\mu_i}
\]
occurs as a subword of \(\rho(w)\).

Then, for every \(1\le i\le r\), the following are equivalent:
\begin{enumerate}
\item[(i)] the subpath \([\nu_{i-2},\nu_{i+2}]\) passes through \(\nu_i\) exactly once;
\item[(ii)] neither \([\nu_{i-1},\nu_i]\) nor \([\nu_i,\nu_{i+1}]\) contains a vertex adjacent to \(\nu_i\) other than \(\mu_i\).
\end{enumerate}
Here, for notational convenience, we set \(\nu_{-1}=\nu_0\) and \(\nu_{r+2}=\nu_{r+1}\).
\end{lem}

\begin{proof}
Fix \(i\), and write \(\nu=\nu_i\) and \(\mu=\mu_i\).

First assume that \([\nu_{i-2},\nu_{i+2}]\) passes through \(\nu\) exactly once. Suppose that \([\nu_i,\nu_{i+1}]\) contains a vertex \(\eta\neq \mu\) adjacent to \(\nu\). Since \(\rho(w)\) turns at \(\nu\), the path contains the segment
\[
\mu \to \nu \to \mu.
\]
On the other hand, by assumption, the subpath from \(\nu\) to \(\nu_{i+1}\) also meets a neighbor \(\eta\) of \(\nu\) with \(\eta\neq \mu\). Because \(\mathcal G\) is a tree, the unique simple path from \(\mu\) to \(\eta\) passes through \(\nu\). It follows that the subpath \([\nu_{i-2},\nu_{i+2}]\) must pass through \(\nu\) at least twice, contrary to (i). Thus \([\nu_i,\nu_{i+1}]\) contains no vertex adjacent to \(\nu\) other than \(\mu\). The same argument, applied to the reversed path, shows that \([\nu_{i-1},\nu_i]\) also contains no vertex adjacent to \(\nu\) other than \(\mu\). Hence (ii) holds.

Conversely, assume that (ii) holds. Suppose that \([\nu_{i-2},\nu_{i+2}]\) contains a second occurrence of \(\nu\). Then this second occurrence must lie either in \([\nu_{i-2},\nu_{i-1})\) or in \((\nu_{i+1},\nu_{i+2}]\). By symmetry, it suffices to consider the latter case.

Thus assume that \(\nu\) occurs again in \((\nu_{i+1},\nu_{i+2}]\). Since \(\mathcal G\) is a tree, any subpath leaving \(\nu\) and later returning to \(\nu\) without meeting \(\nu\) in between must leave and return through the same neighbor of \(\nu\). Because \(\rho(w)\) turns at \(\nu\), that neighbor is precisely \(\mu\). Therefore the segment from the first occurrence of \(\nu\) to the second contains a subword of the form
\[
s_\mu s_\nu s_\mu\, v\, s_\mu s_\nu
\]
for some word \(v\).

We claim that \(v\) contains no generators corresponding to vertices adjacent to \(\nu\). Indeed, if \(v\) contained such a generator, then the corresponding subpath would meet a vertex adjacent to \(\nu\) somewhere in \([\nu_i,\nu_{i+1}]\) or in \([\nu_{i+1},\nu_{i+2}]\). By assumption (ii), this vertex could only be \(\mu\). But then the path would return to the branch through \(\mu\) before reaching the second occurrence of \(\nu\), forcing an additional occurrence of \(\nu\) or an extra turn between \(\nu_i\) and \(\nu_{i+2}\), a contradiction.

Hence \(v\) contains no generators adjacent to \(s_\nu\). It follows that \(\rho(w)\) contains a forbidden subword of the form excluded by Lemma~\ref{lem:uglw}, contradicting the fact that \(\rho(w)\) is a u.l.g. Therefore no second occurrence of \(\nu\) can exist, and so \([\nu_{i-2},\nu_{i+2}]\) passes through \(\nu_i\) exactly once. This proves (i).
\end{proof}

 \begin{ex}
 Consider the Weyl group $D_4$ that is \[W=\{s_0,s_1,s_2,s_3\mid \;s_k^2=1,\; [s_i,s_j]=1\; \text{and}\;  (s_is_0)^3=1\text{ for }i,j>0\}.\]
 Its Dynkin diagram $\mathcal{G}$ is (here $\alpha_i$ correspond to $s_i$)
{\tiny \[
 \xymatrix{
  & \alpha_2\ar@{-}[d] & \\
 \alpha_1 \ar@{-}[r]  &    \alpha_0   &  \alpha_3\ar@{-}[l]
 }
 \]}

 Consider the reduced word
  $\rho(w)=s_0s_1s_0s_2 s_0s_3s_0s_1s_0$.
 It corresponds to the path of type III
 \[\rho(w): \alpha_0 \to \alpha_1 \to \alpha_0 \to  \alpha _2 \to \alpha_0 \to \alpha_3 \to \alpha_0 \to \alpha_1 \to \alpha_0
 \]
 with $r=4$ and turning vertices $(\alpha_1,\alpha_2,\alpha_3,\alpha_1)$.
 
 The reduced expression graph taken modulo $C_1$-equivalence classes 
 (two representatives of $C_1$-equivalence classes are connected by a directed edge labeled by $i$ if
 the first class is obtained from the second by applying $s_0s_is_0 \to s_i s_0 s_i$) 
 for $w$ is given by  
 {\tiny \[
 \xymatrix{
 & 010230310  \ar[rd]|{1}& 102301030 \ar[d]|{1}\ar[r]|{3} & 102301303 &  & \\
 & 101203010 \ar[r]|{3} \ar[d]|{1} & 101230310 & 102013031\ar[d]|{2} & 102030103 \ar[l]|{1} \ar[r]|{2} \ar[lu]|{3}& 120230103 \ar[lld]|{1}\\
 010203010 \ar[ru]|{1} \ar[ruu]|{3}  \ar[dr]|{1} \ar[rdd]|{2} & 101203101 & 102010301 \ar[l]|{1} \ar[ru]|{3} \ar[rd]|{2}& 120213031   & &\\
    & 010203101 \ar[r]|{2}\ar[u]|{1} & 012023101 & 120210301 \ar[u]|{3} & 201020301 \ar[l]|{1}\ar[r]|{3}\ar[dl]|{2} & 201023031 \ar[llu]|{1} \\
  & 012023010  \ar[ru]|{1}& 020102301 \ar[u]|{1}\ar[r]|{2} & 202102301& &
 }
 \]}
 Since there is only one reduced expression with 5 generators $s_0$, the reduced word $\rho(w)$
 is a \ulg\  for the label $(5,2,1,1)$.
 
 Observe that the word $w=(s_0s_1s_0s_2s_0s_3)^2s_0$ is not reduced.
Indeed, we have
{\tiny
\begin{align*}
0102030102030 \to & (101)2(303)1(202)30 \to 1023(101)230230 \to 1023010230230 \to \\
& 1023010203020 \to 1023012023202 \to 10230120302.
\end{align*}
}
This example is consistent with both (i) and (ii) of Lemma~\ref{lem:symmetric-local}. 
Indeed, each turning vertex \(\alpha_i\) is adjacent only to the central vertex \(\alpha_0\). 
 \end{ex}

\begin{dfn}
Let \(\mathcal{G}\) be a simply laced tree. The \emph{valency} of a vertex of \(\mathcal{G}\) is the number of vertices adjacent to it. A vertex is called \emph{branching} if its valency is at least~\(3\), and an \emph{end vertex} if its valency is equal to~\(1\).

A \emph{branch} of \(\mathcal{G}\) is a maximal connected subgraph which is a chain, i.e., all of its vertices have valency at most~\(2\).
\end{dfn}

\begin{dfn}
Let \(\rho(w)\) be a \ulg, and let \((\nu_1,\ldots,\nu_r)\) be the list of its turning vertices (excluding the initial and terminal vertices), so that \(\rho(w)\) contains subwords of the form \(s_{\mu_i}s_{\nu_i}s_{\mu_i}\) for \(1\le i\le r\).

If the vertex \(\mu_i\) is branching, then \(\nu_i\) is called a \emph{short turning vertex}; otherwise, it is called a \emph{long turning vertex}.
\end{dfn}

\begin{dfn}\label{def:finbr}
The \emph{branching index} of \(\rho(w)\) with respect to \(\mathcal{G}\) is defined as the number (counted with multiplicity) of short turning vertices, that is,
\[
\br(\rho(w))
=
\bigl|\{\, i\in\{1,\ldots,r\}\mid \mu_i \text{ is a branching vertex of } \mathcal{G}\,\}\bigr|.
\]
If \(\rho(w)\) has no turning vertices, we set \(\br(\rho(w))=0\).
\end{dfn}

 As an immediate consequence of Corollary~\ref{cor:typesulg}, we obtain the following.

\begin{lem}
Suppose a \ulg\ \(\rho(w)\) visits a branch of the tree \(\mathcal{G}\), that is, it enters and exits the branch through the branching vertex \(\gamma\) to which the branch is attached. Then \(\rho(w)\) has exactly one turning vertex in that branch.

Equivalently, each visit to a branch corresponds to a unique turning vertex contained in that branch.
\end{lem}

\begin{proof}
Since a branch is a chain, the restriction of \(\rho(w)\) to that branch is a path in type \(A\). Because \(\rho(w)\) enters and exits the branch through the same branching vertex \(\gamma\), it must reverse direction somewhere inside the branch, and hence it has at least one turning vertex there. On the other hand, by Corollary~\ref{cor:typesulg}, a \ulg\ cannot have more than one turn along a chain segment between two consecutive passages through the same vertex. Therefore there is exactly one turning vertex in the branch.
\end{proof}

\begin{lem}
Suppose a \ulg\ \(\rho(w)\) visits the same branch of \(\mathcal{G}\) more than once. Let \((\gamma_1,\ldots,\gamma_s)\), with \(s>1\), be the corresponding list of turning vertices on that branch (so this is a sublist of the turning vertices of \(\rho(w)\)).

Let \(d_i\) denote the distance from \(\gamma_i\) to the branching vertex \(\gamma\) of the branch. (In particular, a short turning vertex has distance \(1\), while a long turning vertex has distance \(>1\).)

Then for any \(i\neq j\), the subpath \([\gamma_i,\gamma_j]\) contains a turning vertex \(\gamma_k\) such that
\[
d_k=\max\bigl(1,\min(d_i,d_j)-1\bigr).
\]
\end{lem}

\begin{proof}
Suppose, to the contrary, that there exist \(i\neq j\) such that the subpath \([\gamma_i,\gamma_j]\) contains no turning vertex \(\gamma_k\) with
\[
d_k=\max\bigl(1,\min(d_i,d_j)-1\bigr).
\]
Without loss of generality, assume \(i<j\), and set \(a=\gamma_i\), \(b=\gamma_j\), with distances \(d_a\) and \(d_b\), respectively.

Since the branch is a chain, the path between \(a\) and \(b\) must move monotonically toward the branching vertex \(\gamma\) and then away from it. The absence of a turning vertex at distance \(\min(d_a,d_b)-1\) implies that the path does not turn at the expected intermediate level. In particular, if \(d_a>1\) and \(d_b>1\), then all turning vertices between \(a\) and \(b\) lie strictly closer to \(\gamma\) than \(\min(d_a,d_b)-1\).

It follows that the corresponding subword of \(\rho(w)\) contains a segment of the form
\[
s_\mu s_\nu s_\mu\, v\, s_\mu s_\nu,
\]
where \(v\) involves only generators strictly closer to \(\gamma\), and hence none adjacent to \(s_\nu\). This is a subword of the type excluded by Lemma~\ref{lem:uglw}, a contradiction.

Therefore such a pair \(i,j\) cannot exist, and the claim follows.
\end{proof}

\begin{cor}\label{cor:shortlong}
Let \(\rho(w)\) be a \ulg\ and let \((\gamma_1,\ldots,\gamma_s)\) be the list of turning vertices on a fixed branch. If \(\gamma_i\) is a long turning vertex, then at least one of its neighbors, \(\gamma_{i-1}\) or \(\gamma_{i+1}\), is a short turning vertex.
\end{cor}

\begin{proof}
Let \(d_i\) denote the distance from \(\gamma_i\) to the branching vertex \(\gamma\). Since \(\gamma_i\) is long, we have \(d_i>1\).

Apply the previous lemma to the pairs \((\gamma_{i-1},\gamma_i)\) and \((\gamma_i,\gamma_{i+1})\). It follows that each of the subpaths \([\gamma_{i-1},\gamma_i]\) and \([\gamma_i,\gamma_{i+1}]\) must contain a turning vertex at distance
\[
\max(1,d_i-1)=d_i-1.
\]
Such a turning vertex can only be \(\gamma_{i-1}\) or \(\gamma_{i+1}\), respectively. Hence at least one of \(\gamma_{i-1}\) or \(\gamma_{i+1}\) has distance \(1\), i.e., is a short turning vertex.
\end{proof}

Consider a \ulg\ \(\rho(w)\) of maximal length with trivial branching index, that is, \(\mathrm{br}(\rho(w))=0\).

\begin{lem}
The path \(\rho(w)\) visits every vertex of the tree \(\mathcal{G}\).
\end{lem}

\begin{proof}
Suppose, for a contradiction, that \(\rho(w)\) does not visit some vertex \(v\) of \(\mathcal{G}\). Since \(\mathcal{G}\) is connected, there exists a vertex \(x\) adjacent to \(v\) that is visited by \(\rho(w)\).

Because \(\mathcal{G}\) is a tree, removing \(v\) separates \(\mathcal{G}\) into two connected components. The path \(\rho(w)\), which does not pass through \(v\), lies entirely in one of these components and therefore passes through \(x\) without ever entering the other component.

Write \(\rho(w)=w_1\, x\, w_2\), where the path passes through \(x\) at least once. Then we can construct a longer path by inserting the detour
\[
x \to v \to x,
\]
that is, by replacing the segment \(w_1\, x\, w_2\) with
\[
w_1\, x\, v\, x\, w_2.
\]
Since \(v\) is adjacent to \(x\), this produces a valid path in \(\Gamma\), and it preserves the property that no short turning vertices are introduced (as \(v\) is not branching). Thus we obtain a \ulg\ of strictly greater length, contradicting the maximality of \(\rho(w)\).
Therefore \(\rho(w)\) must visit every vertex of \(\mathcal{G}\).
\end{proof}

\begin{cor}
All turning vertices of \(\rho(w)\) have valency \(1\). In particular, each end vertex is visited exactly once.
\end{cor}

\begin{proof}
Since \(\mathrm{br}(\rho(w))=0\), all turning vertices are long. By definition, if \(\nu_i\) is a turning vertex with associated neighbor \(\mu_i\), then \(\mu_i\) is not a branching vertex, hence has valency at most \(2\). As \(\mu_i\) is adjacent to \(\nu_i\), it follows that \(\nu_i\) cannot have valency greater than \(1\); otherwise \(\mu_i\) would be adjacent to at least two vertices distinct from \(\nu_i\), making it a branching vertex, a contradiction. Thus every turning vertex has valency \(1\).

Finally, an end vertex cannot be visited more than once: otherwise the path would have to leave and return along the unique incident edge, creating more than one turn at that vertex, which is impossible for a \ulg.
\end{proof}

\begin{cor}
Let \((\nu_1,\ldots,\nu_r)\) be the turning vertices of \(\rho(w)\), and let \(d(\nu_i,\nu_{i+1})\) denote the distance in \(\mathcal{G}\) between consecutive turning vertices. Then the sum
\[
S=\sum_{i=1}^{r-1} d(\nu_i,\nu_{i+1})
\]
is maximal among all \ulg’s with trivial branching index.
\end{cor}

\begin{proof}
By the previous corollary, all turning vertices are end vertices of \(\mathcal{G}\). Since \(\rho(w)\) has maximal length and visits every vertex of \(\mathcal{G}\), the path between consecutive turning vertices must follow the unique simple path in the tree.

Suppose that for some \(i\), the distance \(d(\nu_i,\nu_{i+1})\) is not maximal, i.e., the path between \(\nu_i\) and \(\nu_{i+1}\) does not traverse the entire branch between them. Then one can extend \(\rho(w)\) by inserting an additional detour along unused vertices of that branch, increasing the length of the path while preserving the property \(\mathrm{br}(\rho(w))=0\). This contradicts the maximality of \(\rho(w)\).

Therefore each segment between consecutive turning vertices realizes the full distance in \(\mathcal{G}\), and the sum \(S\) is maximal.
\end{proof}

\begin{lem}\label{lem:leafsum}
Let \(\mathcal{G}\) be a tree with \(n\) vertices, and let
\(\lambda_1,\dots,\lambda_m\) be distinct end vertices. Then
\[
\sum_{i=1}^{m-1} d(\lambda_i,\lambda_{i+1}) \le \binom{n}{2}.
\]
\end{lem}

\begin{proof}
We argue by induction on \(n\). The statement is clear for \(n\le 3\).
Let \(n\ge 4\) and remove the last leaf \(\lambda_m\). The remaining graph is a tree on \(n-1\) vertices, and \(\lambda_1,\dots,\lambda_{m-1}\) remain distinct leaves. By the induction hypothesis,
\[
\textstyle
\sum_{i=1}^{m-2} d(\lambda_i,\lambda_{i+1}) \le \binom{n-1}{2}.
\]
Since the diameter of a tree on \(n\) vertices is at most \(n-1\), we also have
\[
d(\lambda_{m-1},\lambda_m)\le n-1.
\]
Adding the two inequalities yields
\[
\textstyle
\sum_{i=1}^{m-1} d(\lambda_i,\lambda_{i+1})
\le \binom{n-1}{2} + (n-1)
= \binom{n}{2},
\]
as required.
\end{proof}

\begin{prop}
Let \(\mathcal{G}\) be a simply laced tree with \(n\ge 4\) vertices, and let \(W\) be the corresponding Coxeter group. Then the maximal length of a \ulg\ \(\rho(w)\) with \(\br(\rho(w))=0\) is bounded above by
\[
\tfrac{1}{2}(n^2 + 5n) - 7.
\]
\end{prop}

\begin{proof}
Let \((\nu_1,\dots,\nu_r)\) be the turning vertices of \(\rho(w)\), and set
\[
S=\sum_{i=1}^{r-1} d(\nu_i,\nu_{i+1}).
\]

By Corollary~\ref{cor:typesulg} and the assumption \(\br(\rho(w))=0\), all turning vertices are long, hence are end vertices of \(\mathcal{G}\), and each is visited at most once. Thus \((\nu_1,\dots,\nu_r)\) is a sequence of distinct leaves, and Lemma~\ref{lem:leafsum} gives
\[
S \le \binom{n}{2}.
\]

Next, by Corollary~\ref{cor:maxchain}, the initial and terminal segments of \(\rho(w)\), which may lie along a maximal chain, contribute at most
\(
3(n-1)-4
\)
to the total length.

Therefore,
\[
\ell(\rho(w))
\le S + 3(n-1)-4
\le \binom{n}{2} + 3(n-1)-4
= \frac{n^2}{2} + \frac{5}{2}n - 7,
\]
as claimed.
\end{proof}

\begin{thm}\label{thm:finitebr}
Let \(\Gamma\) be a simply laced tree with \(n\) vertices. 
Let \(\rho(w)\) be a \ulg\ with branching index \(\br(\rho(w))=B\). 
Then the length of \(\rho(w)\) satisfies
\[
\ell(\rho(w)) \le n^2(B+1)+nB.
\]
In particular, a \ulg\ with finite branching index has finite length.
\end{thm}

\begin{proof}
Let \(V=(\nu_1,\ldots,\nu_r)\) be the list of turning vertices of \(\rho(w)\).

We first bound \(r\).

By Corollary~\ref{cor:shortlong}, every occurrence of a long turning vertex must be separated from the next occurrence of the same vertex by a short turning vertex on the same branch. In particular, between any two consecutive occurrences of a given long turning vertex there must be a short turning vertex.

Since there are at most \(B\) short turning vertices in total (counted with multiplicity), it follows that any given long turning vertex can occur at most \(B+1\) times in \(V\).

Next, observe that the number of distinct vertices in \(\Gamma\) is \(n\). Hence the total number of long turning vertices (counted with multiplicity) is at most \(n(B+1)\), and the number of short turning vertices is exactly \(B\). Therefore,
\[
r \le n(B+1) + B.
\]

Finally, between any two consecutive turning vertices \(\nu_i,\nu_{i+1}\), the path \(\rho(w)\) travels along a simple path in the tree, and hence visits at most \(n\) vertices. Therefore,
\[
\ell(\rho(w)) \le n \cdot r \le n\bigl(n(B+1)+B\bigr)
= n^2(B+1)+nB,
\]
as claimed.
\end{proof}

\section{Infinite uniquely labeled geodesics}

Consider the infinite Coxeter group \(\tilde{D}_6\), whose Coxeter--Dynkin diagram is the simply laced tree
{\tiny \[
\xymatrix{
   & 1 \ar@{-}[d] & 3 \ar@{-}[d] & \\
2 \ar@{-}[r] & a \ar@{-}[r] & b \ar@{-}[r] & 4
}
\]}
In this section, we construct an infinite reduced word in \(\tilde{D}_6\) with the property that every finite connected subword is a \ulg

We note that a similar construction was obtained in \cite{LP13} for the group \(\tilde{A}_2\); however, in that case the Coxeter--Dynkin diagram is not a tree (see Example~\ref{ex:aff}).

We first exhibit a word whose sufficiently large powers are not reduced.

\begin{lem}
Let
\[
P=a1a2ab3b4b.
\]
Then \(P^n\) is not reduced for every \(n\ge 3\).
\end{lem}

\begin{proof}
We use the standard root criterion for reduced expressions in Coxeter groups: a word
\(
s_{i_1}\cdots s_{i_\ell}
\)
is reduced if and only if
\[
s_{i_1}\cdots s_{i_{k-1}}(\alpha_{i_k})
\]
is a positive root for every \(1\le k\le \ell\).

Consider \(P^3\). The \(23\)-rd letter of \(P^3\) is \(a\), and the preceding prefix is
\[
P^2a1.
\]
A direct calculation in the geometric representation of the Coxeter group gives
\[
P^2a1(\alpha_a)=-\alpha_1.
\]
Thus the root corresponding to the \(23\)-rd letter is negative. By the root criterion, \(P^3\) is not reduced.

Since \(P^3\) occurs as an initial subword of \(P^n\) for every \(n\ge 3\), the word \(P^n\) is not reduced for all \(n\ge 3\).
\end{proof}

\begin{ex}
Consider the Coxeter group of type \(D_4\) with Coxeter--Dynkin diagram
{\tiny \[
\xymatrix{
 & 2 \ar@{-}[d] & \\
1 \ar@{-}[r] & 3 \ar@{-}[r] & 4
}
\]}
Let
\[
P = s_3 s_1 s_3 s_2 s_3 s_4 s_3.
\]
This word repeatedly passes through the branching vertex \(3\) and visits each of the three branches.

Although \(P\) itself is reduced, its powers are eventually not reduced. In fact, one checks that \(P^3\) is not reduced. Indeed, by the root criterion for reduced expressions, there exists a position \(k\) in the word \(P^3\) such that
\[
s_{i_1}\cdots s_{i_{k-1}}(\alpha_{i_k})
\]
is a negative root. Consequently, \(P^3\) is not reduced, and hence \(P^n\) is not reduced for all \(n \ge 3\).

This example illustrates that, in contrast to the affine case, such “branch-bouncing’’ words in finite type must eventually fail to be reduced due to the finiteness of the root system.
\end{ex}

Let \(W\) be a Coxeter group with generating set \(S\). In \cite{S09}, the authors reformulate a result of \cite{FZ07} as follows.

\begin{lem}\label{cor:speyer_sets}
Suppose that \(S = I \sqcup J\) is a partition such that the elements of \(I\) commute pairwise, and the elements of \(J\) commute pairwise. If \(W\) is infinite and irreducible, then for any ordering of the products, the word
\[
\Bigl(\prod_{i \in I} s_i \;\prod_{j \in J} s_j\Bigr)^n
\]
is reduced for all \(n \ge 1\).
\end{lem}
We set \(I=\{a,3,4\}\) and \(J=\{b,1,2\}\). We construct an infinite word which is closely related to the one guaranteed to be reduced by Lemma~\ref{cor:speyer_sets}. 

Let
\[
w = a1ab3ba2ab4b.
\]
We will show that the words \(w^n\) are \ulg.

\begin{lem}
The word \(w^2\) can be transformed, using Coxeter relations, into
\[
1a 3 \cdot 1b2 \cdot 3a4 \cdot 2b1 \cdot 4a3 \cdot 1b2 \cdot 3a4 \cdot 2b4,
\]
where the subword
\[
1b2 \cdot 3a4 \cdot 2b1 \cdot 4a3 \cdot 1b2 \cdot 3a4
\]
is reduced.
\end{lem}

\begin{proof}
We begin with
\[
w^2 = (a1ab3ba2ab4b)^2.
\]
Using commutation relations between non-adjacent generators and braid relations for adjacent ones, we rewrite \(w^2\) as
\[
1a1 \cdot 3b3 \cdot 2a2 \cdot 4b4 \cdot 1a1 \cdot 3b3 \cdot 2a2 \cdot 4b4.
\]
Reordering the factors via Coxeter relations yields
\[
1a \cdot 3\,1b \cdot 2\,3 \cdot a \cdot 4\,2 \cdot b \cdot 1\,4 \cdot a \cdot 3\,1 \cdot b \cdot 2\,3 \cdot a4 \cdot 2 \cdot b4,
\]
which can be grouped as
\[
1a3 \cdot 1b2 \cdot 3a4 \cdot 2b1 \cdot 4a3 \cdot 1b2 \cdot 3a4 \cdot 2b4.
\]

By Lemma~\ref{cor:speyer_sets}, the subword
\[
1b2 \cdot 3a4 \cdot 2b1 \cdot 4a3 \cdot 1b2 \cdot 3a4
\]
is reduced, since it is a product of commuting blocks corresponding to the partition \(S=I\sqcup J\).
\end{proof}

\begin{cor}
For all \(n \ge 2\), the word \(w^n\) can be rewritten using Coxeter relations in the form
\[
w^n \;\leadsto\; 1a3\, (1b23a4)^{2n-1}\, 2b4.
\]
In particular, \(\ell(w^n) \ge 12n - 12\).
\end{cor}

\begin{proof}
We first compute \(w^3 = w^2 \cdot w\). Using the expression for \(w^2\) from the previous lemma and applying Coxeter relations, we obtain
\begin{align*}
w^3 
&= \bigl(1a3 \cdot 1b2 \cdot 3a4 \cdot 2b1 \cdot 4a3 \cdot 1b2 \cdot 3a4 \cdot 2b4\bigr)\cdot (a1ab3ba2ab4b) \\
&\leadsto 1a3 \cdot 1b23a42b14a31b23a4 \cdot 2b4 \cdot 1a13b32a24b4 \\
&\leadsto 1a3 \cdot 1b23a42b14a31b23a4 \cdot 2b1 \cdot 4a3 \cdot 1b2 \cdot 3a4 \cdot 2b4 \\
&= 1a3 \cdot (1b23a4)^5 \cdot 2b4.
\end{align*}

Repeating the same argument inductively, we obtain
\[
w^n \;\leadsto\; 1a3 \cdot (1b23a4)^{2n-1} \cdot 2b4.
\]

The resulting word has \(6(2n-1) + 6 = 12n\) letters. The middle part \((1b23a4)^{2n-1}\) is reduced by Lemma~\ref{cor:speyer_sets}. Possible cancellations can occur only at the junctions with the prefix \(1a3\) and suffix \(2b4\), each of length \(3\), and hence at most \(12\) letters can cancel in total. Therefore,
\[
\ell(w^n) \ge 12n - 12,
\]
as claimed.
\end{proof}

\begin{prop}
The word \(w^n = (a1ab3ba2ab4b)^n\) is reduced for all \(n \ge 2\).
\end{prop}

\begin{proof}
Suppose, for a contradiction, that there exists \(p \ge 2\) such that \(w^p\) is not reduced. Then there is a shorter word representing the same group element, of length at most \(12p - 1\).

Consider now \(w^n\) with \(n \gg p\), say \(n \ge 13p\). Write the path in the Cayley graph corresponding to \(w^n\) as a concatenation of \(n\) copies of the path for \(w\). Let \(w^n_-\) and \(w^n_+\) denote the initial and terminal vertices.

By the previous corollary, any reduced expression for \(w^n\) must have length at least \(12n - 12\). On the other hand, if \(w^p\) admits a shortening by at least one letter, then repeating this shortening along the \(n\)-fold concatenation yields a path from \(w^n_-\) to \(w^n_+\) of length at most
\[
12n - \left\lfloor \frac{n}{p} \right\rfloor.
\]
For \(n \ge 13p\), this gives a path of length strictly less than \(12n - 12\), contradicting the lower bound above.

Therefore, \(w^p\) must be reduced for all \(p \ge 2\), and hence \(w^n\) is reduced for all \(n \ge 2\).
\end{proof}

\begin{thm}\label{thm:ulg}
For every \(n\ge 1\), the word
\[
w^n=(a1ab3ba2ab4b)^n
\]
is a uniquely labeled geodesic.
\end{thm}

\begin{proof}
We first note that \(w^n\) is reduced by the preceding proposition. It remains to show that no other reduced expression for the same element has the same label.

Any two reduced expressions for the same element are connected by a sequence of Coxeter moves. Commutation moves preserve the label, while a braid move of the form
\[
s_\alpha s_\beta s_\alpha \longleftrightarrow s_\beta s_\alpha s_\beta
\]
changes the number of occurrences of \(s_\alpha\) and \(s_\beta\). Hence, if another reduced expression has the same label as \(w^n\), then any braid move appearing in such a sequence must eventually be compensated by an opposite braid move.

In the word \(w^n\), all possible braid subwords are of the form
\[
a1a,\quad a2a,\quad b3b,\quad b4b.
\]
Suppose, for instance, that the braid move \(a1a\mapsto 1a1\) occurs. In order to recover the original label, a reverse braid involving the same pair \((a,1)\) must occur later. Between these two braid moves the word contains no generator adjacent to \(1\) other than \(a\), since \(1\) is an end vertex of the Dynkin diagram. Thus the original word would contain a forbidden subword of the form excluded by Lemma~\ref{lem:uglw}. This contradicts the fact that the given expression is reduced and satisfies the local u.l.g. condition.

The same argument applies to the pairs \((a,2)\), \((b,3)\), and \((b,4)\). Therefore no non-trivial braid move can occur in a label-preserving sequence of reduced expressions starting from \(w^n\).

It remains only to consider commutation moves. However, in the word
\[
w^n=(a1ab3ba2ab4b)^n
\]
every pair of consecutive generators is connected by an edge in the Dynkin diagram. Hence no commutation move can be applied to \(w^n\) itself. Since no braid move is possible in a label-preserving sequence, no non-trivial sequence of Coxeter moves preserves the label.

Therefore \(w^n\) is the unique reduced expression with its label. Hence \(w^n\) is a \ulg
\end{proof}

\begin{cor} Consider the words
\[
w_2=a1ab4ba2ab3b,\quad
w_3=a2ab3ba1ab4b,\quad
w_4=a2ab4ba1ab3b.
\]
obtained from \(w_1=w\) by permuting the branches.

Then for any sequence \(i_1,\ldots,i_k\in\{1,2,3,4\}\), the word
\[
w_{i_1}w_{i_2}\cdots w_{i_k}
\]
is also a uniquely labeled geodesic.

Moreover, the number \(N_{\ulg}(R)\) of \ulg's of length at most \(R\) satisfies
\[
N_{\ulg}(R)\ge 1+2\bigl(7^{\lfloor R/12\rfloor}-1\bigr).
\]
In particular, the number of \ulg's grows at least exponentially with rate
$
7^{1/12}$.
\end{cor}

\begin{proof}
Each word \(w_i\) is obtained from \(w_1\) by permuting the four branches attached to the two branching vertices \(a\) and \(b\). 
It remains to check that no new obstruction appears at the junctions between two consecutive power blocks \(w_{i_j}\) and \(w_{i_{j+1}}\). Each block begins with one of the subwords
\[
a1a,\qquad a2a
\]
and ends with one of the subwords
\[
b3b,\qquad b4b.
\]
Thus every junction has the form
\[
\cdots b\epsilon b\, a\delta a\cdots,
\]
where \(\epsilon\in\{3,4\}\) and \(\delta\in\{1,2\}\). Since the generators supported on different branches commute only in the prescribed way, and since the path passes from a branch of \(b\) to a branch of \(a\) without creating a subword of the forbidden form in Lemma~\ref{lem:uglw}, the concatenation remains reduced.

Moreover, any label-preserving Coxeter transformation of the word would have to involve a compensated braid move. Such a compensated braid move would necessarily produce a forbidden subword of the form excluded by Lemma~\ref{lem:uglw}. This is impossible for the same reason as in the proof for $w^n$.

Therefore the concatenated word
\[
w_{i_1}w_{i_2}\cdots w_{i_k}
\]
is reduced and admits no distinct reduced expression with the same label. Hence it is a uniquely labeled geodesic.

Now each word \(w_i\) has length \(12\), and the four words correspond to the four directed edges emanating from a vertex in the embedded \(4\)-regular tree, i.e.\ the Cayley graph of the free group on four generators.

The number of reduced words of block-length \(k\) in a free group on two generators is
$
4\cdot \textcolor{red}{7}^{k-1}$ 
for \(k\ge 1\). Each 
$
w_{i_1}\cdots w_{i_k}
$
is a \ulg\ of length \(12k\). Therefore, for
$
K=\left\lfloor \frac{R}{12}\right\rfloor,
$
we obtain the desired bound
\[
N_{\ulg}(R)
\ge
1+\sum_{k=1}^{K}4\cdot 7^{k-1}
=
1+2(7^K-1).\qedhere
\]
\end{proof}

We conclude with an observation concerning the bi-infinite geodesics arising from the words \(w_1^n\), \(w_2^n\), \(w_3^n\), and \(w_4^n\) in the Cayley graph of \(W\).

Two geodesics \(\gamma_1\) and \(\gamma_2\) are said to be \emph{fellow traveling} if there exists a constant \(D\) such that every vertex of \(\gamma_1\) lies at distance at most \(D\) from some vertex of \(\gamma_2\), and vice versa.

\begin{thm}\label{thm:fourgeodesics}
The bi-infinite uniquely labeled geodesics with labels
\begin{enumerate}
\item \(w_1^n = (a1ab3ba2ab4b)^n\),
\item \(w_2^n = (a1ab4ba2ab3b)^n\),
\item \(w_3^n = (a2ab3ba1ab4b)^n\),
\item \(w_4^n = (a2ab4ba1ab3b)^n\),
\end{enumerate}
are pairwise fellow traveling: there exists a constant \(D \le 5\) such that any two of these geodesics remain within distance at most \(D\) of each other. Moreover, their mutual distance is bounded below by \(2\) along their entire lengths.
\end{thm}

\begin{proof}
We illustrate the argument for \(w_1\) and \(w_2\); the other cases follow by symmetry of the Dynkin diagram.

Consider
\[
w_1^2 = (a1ab3ba2ab4b)^2.
\]
Using Coxeter relations, we rewrite:
{\small
\begin{align*}
w_1^2 
= \; & 1a1 \cdot 3b3 \cdot 2a2 \cdot 4b4 \cdot 1a1 \cdot 3b3 \cdot 2a2 \cdot 4b4 \\
\leadsto \; & 1a3\, 1b2\, 3a4\, 2b1\, 4a3\, 1b2\, 3a4\, 2b4 \\
\leadsto \; & 1a3\, 2b1\, 3a4\, 1b2\, 3a4\, 2b1\, 3a4\, 2b4 \\
\leadsto \; & 1a23b3a1ab4ba2ab3b1a24b4.
\end{align*}
}

The final expression differs from \(w_2^2\) by a cyclic permutation. Thus the endpoints of the paths labeled by \(w_1^2\) and \(w_2^2\) lie at uniformly bounded distance. Counting the number of Coxeter relations used gives an upper bound \(6\); a direct inspection (see Figure~\ref{fig:bundle}) shows that the distance is in fact at most \(4\).

Since the above transformations are local and repeat uniformly along the word, the same bound holds for all powers \(w_1^n\). The same argument, applied to permutations of the branches, shows that each of \(w_3^n\) and \(w_4^n\) is obtained from \(w_1^n\) by a uniformly bounded sequence of local moves. Hence all four bi-infinite geodesics form a bundle of pairwise fellow traveling geodesics with uniform bound \(D \le 4\).

Finally, since the words differ by non-trivial permutations of generators supported on distinct branches, their distance is bounded below by \(2\) along the entire length.
\end{proof}

\begin{figure}[H]
    \centering
    \includegraphics[width=1\linewidth]{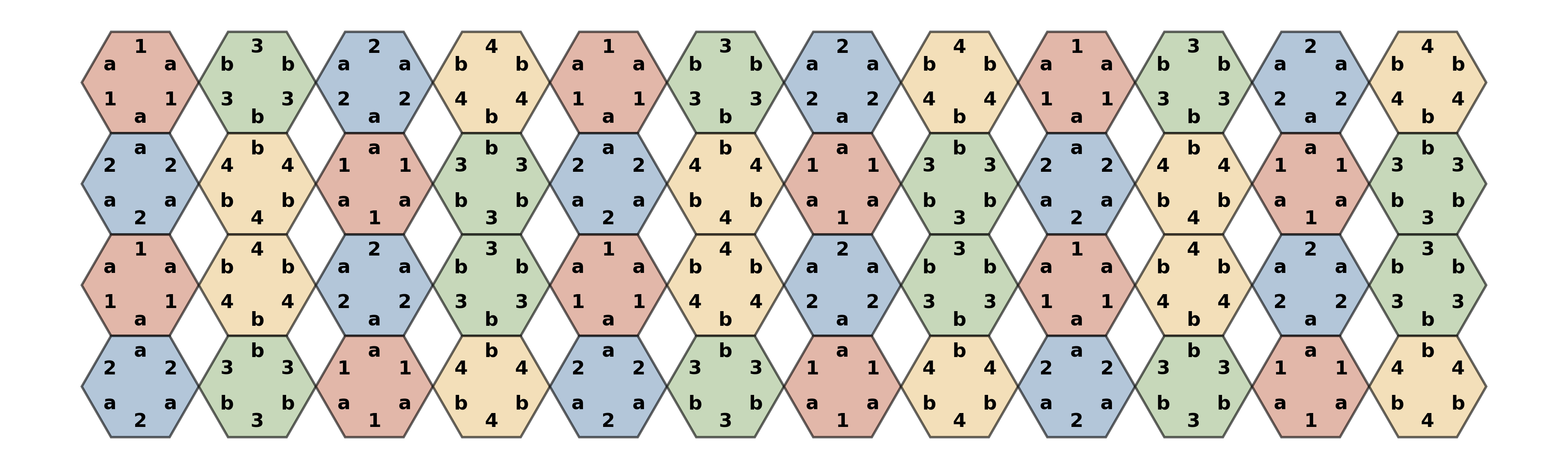}
    \caption{The geodesics $w_1^n$ on the top, $w_4^n$ in the second row, $w_2^n$ in the third row and $w_3^n$ in the bottom row are fellow traveling.}
    \label{fig:bundle}
\end{figure}

\end{document}